\documentclass{amsart}
\usepackage{amsfonts, amsbsy, amsmath, amssymb}

\usepackage{rotating}

\newtheorem{thm}{Theorem}[section]
\newtheorem{lem}[thm]{Lemma}
\newtheorem{cor}[thm]{Corollary}

\newtheorem{conj}[thm]{Conjecture}

\numberwithin{equation}{section}

\begin{document}

\title{A Class of Permutation Trinomials over Finite Fields}

\author[Xiang-dong Hou]{Xiang-dong Hou*}
\address{Department of Mathematics and Statistics,
University of South Florida, Tampa, FL 33620}
\email{xhou@usf.edu}
\thanks{* Research partially supported by NSA Grant H98230-12-1-0245.}



\begin{abstract}
Let $q>2$ be a prime power and $f=-{\tt x}+t{\tt x}^q+{\tt x}^{2q-1}$, where $t\in\Bbb F_q^*$. We prove that $f$ is a permutation polynomial of $\Bbb F_{q^2}$ if and only if one of the following occurs: (i) $q$ is even and $\text{Tr}_{q/2}(\frac 1t)=0$; (ii) $q\equiv 1\pmod 8$ and $t^2=-2$. 
\end{abstract}

\maketitle

\section{Introduction}

Let $\Bbb F_q$ denote the finite field with $q$ elements. A polynomial $f\in\Bbb F_q[{\tt x}]$ is called a {\em permutation polynomial} (PP) of $\Bbb F_q$ if the mapping $x\mapsto f(x)$ is a permutation of $\Bbb F_q$. There is always great interest in permutation polynomials that appear in simple algebraic forms. To this end, a considerable amount of research has been devoted to finding and understanding {\em permutation binomials}. For a few samples from a long list of publications on permutation binomials over finite fields, see 
\cite{Akb-Wan06, Car62, Hou, Kim-Lee95, MPW06, Mas-Zie09, Tur88,Vas-Ryb11,Wan87, Wan94}. As for {\em permutation trinomials}, there are not many theoretic results. Discoveries of infinite classes of permutation trinomials are less frequent than those of permutation binomials \cite{Bri,Gra,Lee-Par97}.

The main results of the present paper are the following theorems.

\begin{thm}\label{T1.1}
Let $q$ be odd and $f=-{\tt x}+t{\tt x}^q+{\tt x}^{2q-1}\in\Bbb F_q[{\tt x}]$, where $t\in\Bbb F_q^*$. Then $f$ is a PP of $\Bbb F_{q^2}$ if and only if $q\equiv 1\pmod 8$ and $t^2=-2$.
\end{thm}

\begin{thm}\label{T1.2}
Let $q>2$ be even and $f={\tt x}+t{\tt x}^q+{\tt x}^{2q-1}\in\Bbb F_q[{\tt x}]$, where $t\in\Bbb F_q^*$. Then $f$ is a PP of $\Bbb F_{q^2}$ if and only if $\text{\rm Tr}_{q/2}(\frac 1t)=0$.
\end{thm}
  
The polynomials in Theorems~\ref{T1.1} and \ref{T1.2} arose from a recent study of certain permutation polynomials over finite fields defined by a functional equation \cite{FHL}. The same study also led to the discovery of a class of permutation binomials over $\Bbb F_{q^2}$ similar to the ones in the above theorems \cite{Hou}. However, the method of \cite{Hou} does not seem to work with the current situation, and we will discuss the reason in Section 4. The method of the present paper is different from that of \cite{Hou}.

The proof of Theorem~\ref{T1.2} is quite easy and will be given in Section 2. 
The proof of Theorem~\ref{T1.1} is rather involved; here we briefly describe the strategy and method for the proof. For the sufficiency part, we try to show that for every $y\in\Bbb F_{q^2}$, the equation $f(x)=y$ has a solution $x\in\Bbb F_{q^2}$. This is quite obvious when $y\in\Bbb F_q$. When $y\in\Bbb F_{q^2}\setminus\Bbb F_q$, the problem is reduced to proving that a certain cubic polynomial over $\Bbb F_q$ is always reducible. For the necessity part, we use the criterion that $\sum_{x\in\Bbb F_{q^2}}f(x)^s=0$ for $1\le s\le q-2$. The sum $\sum_{x\in\Bbb F_{q^2}}f(x)^s$ can be expressed as a double sum in terms of binomial coefficients, and can be made explicit for $s=(q-1)q$ and $1+(q-2)q$. When $s=(q-1)q$, the equation $\sum_{x\in\Bbb F_{q^2}}f(x)^s=0$ implies that $1+4t^{-2}$ is a square in $\Bbb F_q^*$; when $s=1+(q-2)q$, the equation implies that $t^2=-2$.

The proof of Theorem~\ref{T1.1} spans over three sections: Section 3: proof of the sufficiency; Section 4: computation of $\sum_{x\in\Bbb F_{q^2}}f(x)^s$;  Section 5: proof the necessity. 
The polynomial $f$ in Theorems~\ref{T1.1} and \ref{T1.2} arose from a recent work \cite{FHL} on the permutation properties of a class of polynomials defined by a functional equation. We discuss this connection in Section 6. As mentioned above, the power sum $\sum_{x\in\Bbb F_{q^2}}f(x)^s$ equals a double sum involving binomial coefficients. In Section 7, we examine some curious behaviors of that double sum and conclude the paper with a conjecture.  

\medskip
\noindent{\bf Remark.}
The polynomial $f$ in Theorems~\ref{T1.1} and \ref{T1.2} can be written as $f={\tt x}h({\tt x}^{q-1})$, where $h({\tt x})=-1+t{\tt x}+{\tt x}^2$. By \cite[Lemma~2.1]{Zie09}, $f$ is a PP of $\Bbb F_{q^2}$ if and only if ${\tt x}h({\tt x})^{q-1}$ permutes the $(q-1)$st powers in $\Bbb F_{q^2}^*$. However, as described above, our approach does not rely on this observation.

\section{Proof of Theorem~\ref{T1.2}}

Recall that in Theorem~\ref{T1.2}, $q>2$ is even and $t\in\Bbb F_q^*$.

\begin{proof}[Proof of Theorem~\ref{T1.2}]
($\Leftarrow$) Let $y\in \Bbb F_{q^2}$ be arbitrary. We show that the equation
\begin{equation}\label{2.1}
x+tx^q+x^{2q-1}=y
\end{equation}
has at most one solution $x\in \Bbb F_{q^2}$. Assume that $x\in\Bbb F_{q^2}$ is a solution of \eqref{2.1}.

{\bf Case 1.} Assume $y\ne 0$. Then $x\ne 0$. Put $\tau=x^{-q}y=x^{q-1}+x^{1-q}+t\in\Bbb F_q$. Then $x=(\frac y\tau)^q$ and \eqref{2.1} becomes
\[
\Bigl(\frac y\tau\Bigr)^q+t\,\frac y\tau+\Bigl(\frac y\tau \Bigr)^{2-q}=y,
\]
i.e.,
\begin{equation}\label{2.2}
\frac 1\tau(y^{q-1}+y^{1-q}+t)=1.
\end{equation}
Thus $\tau$ is unique, hence so is $x$.

{\bf Case 2.} Assume $y=0$. We show that \eqref{2.1} has no solution $x\in\Bbb F_{q^2}^*$. If, to the contrary, \eqref{2.1} has a solution $x\in\Bbb F_{q^2}^*$, then
\begin{equation}\label{2.3}
x^{q-1}+x^{1-q}+t=0.
\end{equation}
Since $t\ne 0$, we have $x^{2(1-q)}\ne 1$, i.e., $x^{(q-1)^2}\ne 1$. Thus $x^{q-1}\notin\Bbb F_q$. By \eqref{2.3}, $x^{q-1}$ is a root of ${\tt x}^2+t{\tt x}+1\in\Bbb F_q[{\tt x}]$. Then ${\tt x}^2+t{\tt x}+1$ must be irreducible over $\Bbb F_q$. Hence $\text{Tr}_{q/2}(\frac 1t)=1$, which is a contradiction.

($\Rightarrow$) Assume to the contrary that $\text{Tr}_{q/2}(\frac 1t)=1$. Then ${\tt x}^2+t{\tt x}+1$ is irreducible over $\Bbb F_q$. Let $x\in\Bbb F_{q^2}$ be a root of this polynomial. Then
$x^{1+q}=\text{N}_{q^2/q}(x)=1$, hence $x=y^{q-1}$ for some $y\in\Bbb F_{q^2}$. Thus we have $y^{q-1}+y^{1-q}+t=0$. Then \eqref{2.2} does not have any solution for $\tau$. Following the argument in Case 1 of ($\Leftarrow$), we see that $f(x)=y$ has no solution $x\in\Bbb F_{q^2}$, which is a contradiction.
\end{proof}

\section{Proof of Theorem~\ref{T1.1}, Sufficiency}
  
Recall that the discriminant of a cubic polynomial $g={\tt x}^3+b{\tt x}^2+c{\tt x}+d$ over a field is given by
\[
D(g)=-4c^3-27d^2+b^2c^2-4b^3d+18bcd.
\]
We first prove a lemma which appeared as an exercise in \cite[p.53]{Kap}. 

\begin{lem}\label{L3.1}
If $g\in\Bbb F_q[{\tt x}]$ is an irreducible cubic polynomial, then $D(g)$ is a square in $\Bbb F_q^*$.
\end{lem}

\begin{proof}
$\Bbb F_{q^3}$ is the splitting field of $g$ over $\Bbb F_q$. The Galois group $\text{Aut}(\Bbb F_{q^3}/\Bbb F_q)$ of $g$ over $\Bbb F_q$, as a permutation group of the roots of $g$, is $\langle(1,2,3)\rangle=A_3$.  By \cite[Chapter V, Corollary 4.7]{Hun}, $D(g)$ is a square of $\Bbb F_q^*$. 
\end{proof}

\begin{proof}[Proof of Theorem~\ref{T1.1}]
($\Leftarrow$) Recall that $q$ is odd and $t\in \Bbb F_q^*$.
Let $y\in\Bbb F_{q^2}$. We show that equation 
\begin{equation}\label{3.1}
-x+tx^q+x^{2q-1}=y
\end{equation}
has at least one solution $x\in\Bbb F_{q^2}$. If $y\in\Bbb F_q$, $x=(\frac yt)^q$ is a solution. So we assume $y\in\Bbb F_{q^2}\setminus\Bbb F_q$.

Assume for the time being that $x\in\Bbb F_{q^2}$ satisfies \eqref{3.1}. Put $\tau=x^{-q}y=t+x^{q-1}-x^{1-q}$. Since $(x^{q-1}-x^{1-q})^q=-(x^{q-1}-x^{1-q})$,  we may write
\begin{equation}\label{3.3}
\tau=t+\epsilon u,
\end{equation}
where $\epsilon\in\Bbb F_{q^2}$, $\epsilon^{q-1}=-1$, and $u\in\Bbb F_q$. Making the substitution $x=(\frac y\tau)^q$ in \eqref{3.1}, we have
\[
-\Bigl(\frac y\tau \Bigr)^q+t\,\frac y\tau+\Bigl(\frac y\tau \Bigr)^{2-q}=y,
\]
i.e.,
\begin{equation}\label{3.4}
\frac{t-\tau}\tau+y^{1-q}\frac{\tau^q}{\tau^2}-y^{q-1}\frac 1{\tau^q}=0.
\end{equation}
In the light of \eqref{3.3}, we can write \eqref{3.4} as
\begin{equation}\label{3.5}
\frac{-\epsilon u}{t+\epsilon u}+y^{1-q}\frac{t-\epsilon u}{(t+\epsilon u)^2}-y^{q-1}\frac 1{t-\epsilon u}=0.
\end{equation}
A routine computation shows that \eqref{3.5} is equivalent to 
\begin{equation}\label{3.6}
u^3-\frac{y^{q-1}-y^{1-q}}\epsilon u^2+\frac{2-2t(y^{q-1}+y^{1-q})}{\epsilon^2}u+2\,\frac{y^{q-1}-y^{1-q}}{\epsilon^3}=0.
\end{equation}
Note that the left side of \eqref{3.6} is a cubic polynomial in $u$ with coefficients in $\Bbb F_q$.

At this point, it is necessary to reverse the above reasoning to ensure the correctness of the logic. If $u\in\Bbb F_q$ is a solution of \eqref{3.6}, then $\tau=t+\epsilon u$ is a solution of \eqref{3.4}, and consequently, $x=(\frac y\tau)^q$ is a solution of \eqref{3.1}. Therefore, all we have to do is to show that 
\begin{equation}\label{3.7}
g({\tt u}):={\tt u}^3-\frac{y^{q-1}-y^{1-q}}\epsilon{\tt u}^2+\frac{2-2t(y^{q-1}+y^{1-q})}{\epsilon^2}{\tt u}+2\,\frac{y^{q-1}-y^{1-q}}{\epsilon^3}\in\Bbb F_q[{\tt u}]
\end{equation}
is reducible, where $\epsilon\in\Bbb F_{q^2}$, $\epsilon^{q-1}=-1$. (Note that the reducibility of $g$ is independent of the choice of $\epsilon$.)

If $y^{q-1}-y^{1-q}=0$, $g$ is clearly reducible. So we assume $y^{q-1}-y^{1-q}\ne 0$. Since $(y^{q-1}-y^{1-q})^q=-(y^{q-1}-y^{1-q})$, we may choose $\epsilon=y^{q-1}-y^{1-q}$. Let $s=y^{q-1}+y^{1-q}\in\Bbb F_q$. Then 
\[
s^2-4=(y^{q-1}+y^{1-q})^2-4=(y^{q-1}-y^{1-q})^2=\epsilon^2,
\]
which is a nonsquare in $\Bbb F_q^*$ since $\epsilon\notin\Bbb F_q$. We can express $g({\tt u})$ in terms of $t$ and $s$:
\[
g({\tt u})={\tt u}^3-{\tt u}^2+\frac{2-2ts}{s^2-4}{\tt u}+\frac 2{s^2-4}.
\]
We proceed to compute the discriminant of $g$. We have 
\[
\begin{split}
D(g)=\,&-4\Bigl(\frac{2-2ts}{s^2-4}\Bigr)^3-27\Bigl(\frac 2{s^2-4}\Bigr)^2+\Bigl(\frac{2-2ts}{s^2-4}\Bigr)^2+4\cdot\frac 2{s^2-4}\cr
&-18\cdot\frac{2-2ts}{s^2-4}\cdot\frac 2{s^2-4}\cr
=\,&-\frac 4{(s^2-4)^3}\bigl[-8(ts-1)^3+27(s^2-4)-(ts-1)^2(s^2-4)-2(s^2-4)^2\cr
&-18(ts-1)(s^2-4)\bigr]\cr
=\,&-\frac 4{(s^2-4)^3}\bigl[-8(t^3s^3-3t^2s^2+3ts-1)\cr
&+(s^2-4)(27-t^2s^2+2ts-1-18ts+18)-2(s^4-8s^2+16)\bigr]\cr
=\,&-\frac 4{(s^2-4)^3}\bigl[-8(-2ts^3+6s^2+3ts-1)\cr
&+(s^2-4)(2s^2-16ts+44)-2(s^4-8s^2+16)\bigr]\cr
=\,&-\frac{16}{(s^2-4)^3}(s^2+10ts-50)\cr
=\,&-\frac{16(s+5t)^2}{(s^2-4)^3}.
\end{split}
\]
Since $-16(s+5t)^2$ is a square of $\Bbb F_q$ and $s^2-4$ is a nonsquare in $\Bbb F_q^*$, $D(g)$ is not a square in $\Bbb F_q^*$. By Lemma~\ref{L3.1}, $g$ is reducible in $\Bbb F_q[{\tt u}]$.
\end{proof}

\section{Computation of $\sum_{x\in\Bbb F_{q^2}}f(x)^s$}
  
Let $f=-{\tt x}+t{\tt x}^q+t{\tt x}^{2q-1}$, where $t\in\Bbb F_q^*$ and $q>2$. Let $0<s<q^2-1$ and write $s=\alpha+\beta q$, where $0\le \alpha,\beta\le q-1$. We have
\[
\sum_{x\in\Bbb F_{q^2}}f(x)^s=\sum_{x\in\Bbb F_{q^2}^*}f(x)^s=\sum_{x\in\Bbb F_{q^2}^*} x^{\alpha q+\beta}(x^{q-1}-x^{1-q}+t)^{\alpha+\beta q}.
\]
This sum is clearly $0$ when $\alpha+\beta q\not\equiv 0\pmod{q-1}$.

Now assume $\alpha+\beta q\equiv 0\pmod{q-1}$. Since $0<\alpha+\beta q<q^2-1$, we must have $\alpha+\beta=q-1$. The above computation continues as follows:
\[
\kern-2.5cm
\begin{split}
&\sum_{x\in\Bbb F_{q^2}}f(x)^{\alpha+\beta q}\cr
=\,&\sum_{x\in\Bbb F_{q^2}^*}x^{(\alpha+1)(q-1)}(x^{q-1}-x^{1-q}+t)^\alpha(x^{1-q}-x^{q-1}+t)^\beta\cr
\end{split}
\]
\[
\begin{split}
=\,&\sum_{x\in\Bbb F_{q^2}^*}x^{(\alpha+1)(q-1)}\sum_{i,j}\binom \alpha i\binom\beta j(x^{q-1}-x^{1-q})^{i+j}(-1)^jt^{\alpha+\beta-(i+j)}\cr
=\,&\sum_{x\in\Bbb F_{q^2}^*}x^{(\alpha+1)(q-1)}\sum_{i,j}\binom \alpha i\binom\beta j x^{(i+j)(q-1)}(1-x^{2(1-q)})^{i+j}(-1)^jt^{-(i+j)}\cr
=\,&\sum_{x\in\Bbb F_{q^2}^*}\sum_{i,j}\binom \alpha i\binom\beta j x^{(q-1)(\alpha+1+i+j)}(-1)^jt^{-(i+j)}\sum_k\binom{i+j}k(-1)^kx^{-2k(q-1)}\cr
=\,&\sum_{x\in\Bbb F_{q^2}^*}\sum_{i,j,k}\binom \alpha i\binom\beta j\binom{i+j}k(-1)^{k+j}t^{-(i+j)}x^{(q-1)(\alpha+1+i+j-2k)}\cr
=\,&-\sum_{\substack{i,j,k\ge 0\cr \alpha+1+i+j-2k\equiv 0\,(\text{mod}\, q+1)}}\binom \alpha i\binom\beta j\binom{i+j}k(-1)^{k+j}t^{-(i+j)}.
\end{split}
\]
It is easy to see that when $0\le i\le\alpha$, $0\le j\le \beta$, and $0\le k\le i+j$, we have
\[
-(q+1)<\alpha+1+i+j-2k<2(q+1).
\]
Therefore,
\begin{equation}\label{4.1}
\sum_{x\in\Bbb F_{q^2}}f(x)^{\alpha+\beta q}=-\sum_{\substack{i,j,k\ge 0\cr \alpha+1+i+j-2k= 0,\, q+1}}\binom \alpha i\binom\beta j\binom{i+j}k(-1)^{k+j}t^{-(i+j)}.
\end{equation}
The sum at the right side of \eqref{4.1} will go through some further cosmetic changes in Section 7, but its present form is suitable for the proof of the necessity part of Theorem~\ref{T1.1}.  

\medskip
\noindent{\bf Remark.}
In \cite{Hou}, binomials of the form $h=t{\tt x}+{\tt x}^{2q-1}$, where $t\in\Bbb F_q^*$, were considered. The power sum $\sum_{x\in\Bbb F_{q^2}}h(x)^{\alpha+\beta q}$, where $\alpha,\beta\ge 0$ and $\alpha+\beta=q-1$, was expressed in terms of two sums involving binomial coefficients; see \cite[Lemma~3.2]{Hou}. Those two sums in \cite{Hou} are over one variable and essentially depend only on $\alpha$ but not on $q$ (hence not on $\beta$). However,  the sum in \eqref{4.1} is over two variables and depends on both $\alpha$ and $q$. These differences are the reason that the characteristic free approach in \cite{Hou} has not woked for the current situation. 

\section{Proof of Theorem~\ref{T1.1}, Necessity}

Let $p=\text{char}\,\Bbb F_q$ and let $\Bbb Z_p$ be the ring of $p$-adic integers. For $z\in\Bbb Z_p$ and $a\in\Bbb Z$ with $a\ge 0$, written in the form $z=\sum_{n\ge 0}z_np^n$, $a=\sum_{n\ge 0}a_np^n$,
$0\le z_n,a_n\le p-1$, we have
\[
\binom za\equiv\prod_{n\ge 0}\binom{z_n}{a_n}\pmod p.
\]
Therefore, if $z,z'\in\Bbb Z_p$ such that $\nu_p(z-z')>\log_pa$, where $\nu_p$ is the $p$-adic order, we have $\binom za\equiv\binom{z'}a\pmod p$.

\begin{lem}\label{L5.1}
Let $z\in\Bbb F_q^*$ and write ${\tt x}^2+{\tt x}-z=({\tt x}-r_1)({\tt x}-r_2)$, $r_1,r_2\in\Bbb F_{q^2}$. Then
\[
\sum_{0\le k\le\frac q2}\binom{-k}k z^k=
\begin{cases}
\displaystyle\frac 12&\text{if}\ r_1=r_2\in\Bbb F_q, \vspace{2mm}\cr
1&\text{if}\ r_1,r_2\in\Bbb F_q,\ r_1\ne r_2, \vspace{1mm} \cr
0&\text{if}\ r_1,r_2\notin\Bbb F_q.
\end{cases}
\]
\end{lem}

\begin{proof}
We denote the constant term of a Laurent series in ${\tt x}$ by $\text{ct}(\ )$. We have
\begin{equation}\label{5.2}
\begin{split}
\sum_{0\le k\le\frac q2}\binom{-k}k z^k\,&=\sum_{0\le k\le q-1}\binom{-k}k z^k\cr
&\kern 0.8cm \textstyle \text{(when $\frac q2<k\le q-1$, $\binom{-k}k\equiv\binom{q-k}k=0\pmod p$)}\cr
&=\sum_{0\le k\le q-1}\text{ct}\Bigl(\frac 1{{\tt x}^k(1+{\tt x})^k}\Bigr)\cdot z^k\cr
&=\text{ct}\biggl[\,\sum_{0\le k\le q-1}\Bigl(\frac z{{\tt x}(1+{\tt x})}\Bigr)^k\,\biggr]\cr
&=\text{ct}\biggl(\frac{1-(\frac z{{\tt x}(1+{\tt x})})^q}{1-\frac z{{\tt x}(1+{\tt x})}}\biggr)\cr
&=\text{ct}\Bigl[\frac{{\tt x}(1+{\tt x})}{{\tt x}(1+{\tt x})-z}\Bigl(1-z\Bigl(\frac 1{\tt x}-\frac 1{1+{\tt x}}\Bigr)^q\Bigr)\Bigr]\cr
&=\text{ct}\Bigl(\frac{{\tt x}(1+{\tt x})}{{\tt x}(1+{\tt x})-z}\cdot\frac{-z}{{\tt x}^q}\Bigr)\cr
&=\text{ct}\Bigl[\Bigl(1+\frac z{{\tt x}(1+{\tt x})-z}\Bigr)\cdot\frac{-z}{{\tt x}^q}\Bigr]\cr
&=\text{ct}\Bigl(\frac{-z^2}{{\tt x}^q}\cdot\frac 1{{\tt x}^2+{\tt x}-z}\Bigr)\cr
&=\text{ct}\Bigl(\frac{-z^2}{{\tt x}^q}\cdot\frac 1{({\tt x}-r_1)({\tt x}-r_2)}\Bigr).
\end{split}
\end{equation}

First assume $r_1=r_2$. We must have $r_1=-\frac 12$. By \eqref{5.2}, we have 
\[
\begin{split}
\sum_{0\le k\le\frac q2}\binom{-k}k z^k\,&=\text{ct}\Bigl(\frac{-z^2}{{\tt x}^q}({\tt x}-r_1)^{-2}\Bigr)\cr
&=\text{ct}\Bigl(\frac{-z^2}{{\tt x}^q} \,r_1^{-2}\Bigl(1-\frac{\tt x}{r_1}\Bigr)^{-2}\Bigr)\cr
&=-\frac{z^2}{r_1^2}\binom{-2}q\Bigl(-\frac 1{r_1}\Bigr)^q\cr
&=-\frac{z^2}{r_1^{q+2}} \kern 3cm \textstyle{(\binom{-2}q\equiv -1\pmod p)}\cr
&=-\frac{z^2}{r_1^3}\cr
&=-r_1\kern 3.5cm (-z=r_1^2)\cr
&=\frac 12.
\end{split}
\]

Now assume $r_1\ne r_2$. By \eqref{5.2}, we have
\[
\begin{split}
\sum_{0\le k\le\frac q2}\binom{-k}k z^k\,&=\text{ct}\Bigl[\frac{-z^2}{{\tt x}^q}\Bigl(\frac 1{{\tt x}-r_1}-\frac 1{{\tt x}-r_2}\Bigr)\frac 1{r_1-r_2}\Bigr]\cr
&=\text{ct}\Bigl[\frac{z^2}{r_1-r_2}\cdot\frac 1{{\tt x}^q}\Bigl(\frac 1{r_1}\cdot\frac 1{1-\frac{\tt x}{r_1}}-\frac 1{r_2}\cdot\frac 1{1-\frac{\tt x}{r_2}}\Bigr)\Bigr]\cr
&=\frac{z^2}{r_1-r_2}\Bigl(\frac 1{{r_1}^{q+1}}-\frac 1{{r_2}^{q+1}}\Bigr)\cr
&=\frac{-z^2}{(r_1r_2)^{q+1}}\cdot\frac{r_1^{q+1}-r_2^{q+1}}{r_1-r_2}\cr
&=-\frac{r_1^{q+1}-r_2^{q+1}}{r_1-r_2}\cr
&=
\begin{cases}
\displaystyle -\frac{r_1^2-r_2^2}{r_1-r_2}=-(r_1+r_2)=1&\text{if}\ r_1,r_2\in\Bbb F_q, \vspace{2mm}\cr
0&\text{if $r_1,r_2$ are conjugates over $\Bbb F_q$}.
\end{cases} 
\end{split}
\]
This completes the proof of Lemma~\ref{L5.1}.
\end{proof}

\begin{lem}\label{L5.2}
Let $z\in\Bbb F_q^*$ and assume that ${\tt x}^2+{\tt x}-z$ has two distinct roots in $\Bbb F_q$. Then
\[
\sum_{0\le k\le\frac {q-1}2}\binom{-k}{k+1} z^k=1.
\]
\end{lem}

\begin{proof}
The proof is similar to that of Lemma~\ref{L5.1}. Write ${\tt x}^2+{\tt x}-z=({\tt x}-r_1)({\tt x}-r_2)$, where $r_1,r_2\in\Bbb F_q$, $r_1\ne r_2$. We have
\[
\begin{split}
&\sum_{0\le k\le\frac {q-1}2}\binom{-k}{k+1} z^k\cr
=\,&\sum_{0\le k\le q-1}\binom{-k}{k+1} z^k-\binom{-(q-1)}q z^{q-1}\cr
&\kern 2cm \textstyle\text{(when $\frac {q-1}2<k<q-1$, $\binom{-k}{k+1}\equiv 0\pmod p$)}\cr 
=\,&\sum_{0\le k\le q-1}\text{ct}\Bigl(\frac 1{{\tt x}^{k+1}(1+{\tt x})^k}\Bigr)\cdot z^k+1\cr
& \kern 2cm \textstyle\binom{-(q-1)}q\equiv\binom{q^2-(q-1)}q=\binom{1+(q-1)q}q\equiv\binom{q-1}1\equiv -1\pmod p)\cr
=\,&\text{ct}\biggl[\,\frac 1{\tt x}\sum_{0\le k\le q-1}\Bigl(\frac z{{\tt x}(1+{\tt x})}\Bigr)^k\,\biggr]+1\cr
=\,&\text{ct}\Bigl[\frac 1{\tt x}\cdot\frac{{\tt x}(1+{\tt x})}{{\tt x}(1+{\tt x})-z}\Bigl(1-\frac z{{\tt x}^q}+\frac z{(1+{\tt x})^q}\Bigr)\Bigr]+1\cr
=\,&\text{ct}\Bigl[\frac{1+{\tt x}}{{\tt x}(1+{\tt x})-z}\Bigl(1+\frac z{(1+{\tt x})^q}\Bigr)-\frac z{{\tt x}^{q+1}}\cdot \frac{{\tt x}(1+{\tt x})}{{\tt x}(1+{\tt x})-z}\Bigr]+1\cr
=\,&-\frac 1z(1+z)+\text{ct}\Bigl[\frac {-z}{{\tt x}^{q+1}}\Bigl(1+\frac z{{\tt x}(1+{\tt x})-z}\Bigr)\Bigr]+1\cr
\end{split}
\]
\[
\begin{split}
=\,&-\frac 1z+\text{ct}\Bigl(\frac {-z^2}{{\tt x}^{q+1}}\cdot\frac 1{({\tt x}-r_1)({\tt x}-r_2)}\Bigr)\cr
=\,&-\frac 1z+\frac{z^2}{r_1-r_2}\Bigl(\frac 1{r_1^{q+2}}-\frac 1{r_2^{q+2}}\Bigr)\cr
&\kern 2cm \text{(see the computation in the proof of Lemma~\ref{L5.1})}\cr
=\,&-\frac 1z+\frac{z^2}{r_1-r_2}\Bigl(\frac 1{r_1^3}-\frac 1{r_2^3}\Bigr)\cr
=\,&-\frac 1z-\frac{z^2}{(r_1r_2)^3}\cdot \frac{r_1^3-r_2^3}{r_1-r_2}\cr
=\,&-\frac 1z+\frac 1z(r_1^2+r_1r_2+r_2^2)\kern 2cm (r_1r_2=-z)\cr
=\,&-\frac 1z+\frac 1z\bigl[(r_1+r_2)^2-r_1r_2\bigr]\cr
=\,&-\frac 1z+\frac 1z(1+z)\cr
=\,&1.
\end{split}
\]
\end{proof}

\begin{lem}\label{L5.3}
Let $z\in\Bbb F_q^*$ and assume that ${\tt x}^2+{\tt x}-z$ has two distinct roots in $\Bbb F_q$. Then
\[
\sum_{0\le k\le\frac {q-1}2}(k+1)\binom{-k}{k+1} z^k=\frac{2z}{1+4z}.
\]
\end{lem}

\begin{proof}
We have
\begin{equation}\label{5.3}
\begin{split}
\sum_{0\le k\le\frac {q-1}2}(k+1)\binom{-k}{k+1} z^k\,&=\sum_{0\le k\le q-2}(k+1)\binom{-k}{k+1} z^k\cr
&=\sum_{0\le k\le q-2}(k+1)\cdot \text{ct}\Bigl(\frac 1{{\tt x}^{k+1}(1+{\tt x})^k}\Bigr)\cdot z^k\cr
&=\text{ct}\biggl[\,\frac 1{\tt x}\sum_{0\le k\le q-2}(k+1)\Bigl(\frac z{{\tt x}(1+{\tt x})}\Bigr)^k\,\biggr].
\end{split}
\end{equation}
Since 
\[
\sum_{1\le k\le q-1}k{\tt y}^{k-1}=\frac d{d{\tt y}}\sum_{0\le k\le q-1}{\tt y}^k=\frac d{d{\tt y}}\Bigl(\frac{1-{\tt y}^q}{1-{\tt y}}\Bigr)=\frac{1-{\tt y}^q}{(1-{\tt y})^2},
\]
we have 
\begin{equation}\label{5.4}
\begin{split}
\sum_{0\le k\le q-2}(k+1)\Bigl(\frac z{{\tt x}(1+{\tt x})}\Bigr)^k\,&=\frac{1-(\frac z{{\tt x}(1+{\tt x})})^q}{(1-\frac z{{\tt x}(1+{\tt x})})^2}\cr
&=\Bigl(\frac{{\tt x}(1+{\tt x})}{{\tt x}(1+{\tt x})-z}\Bigr)^2\Bigl(1-\frac z{{\tt x}^q}+\frac z{(1+{\tt x})^q}\Bigr).
\end{split}
\end{equation}
Therefore, by \eqref{5.3} and \eqref{5.4},
\eject
\ \vskip-8mm
\[
\begin{split}
&\sum_{0\le k\le\frac {q-1}2}(k+1)\binom{-k}{k+1} z^k\cr
=\,&\text{ct}\Bigl[\frac 1{\tt x}\Bigl(\frac{{\tt x}(1+{\tt x})}{{\tt x}(1+{\tt x})-z}\Bigr)^2\Bigl(1-\frac z{{\tt x}^q}+\frac z{(1+{\tt x})^q}\Bigr)\Bigr]\cr
=\,&\text{ct}\Bigl[-\frac z{{\tt x}^{q+1}}\Bigl(\frac{{\tt x}(1+{\tt x})}{{\tt x}(1+{\tt x})-z}\Bigr)^2\Bigr]\cr
=\,&\text{ct}\Bigl[-\frac z{{\tt x}^{q+1}}\Bigl(1+\frac z{{\tt x}(1+{\tt x})-z}\Bigr)^2\Bigr]\cr
=\,&\text{ct}\Bigl[-\frac z{{\tt x}^{q+1}}\Bigl(1+\frac z{r_1-r_2}\Bigl(\frac 1{{\tt x}-r_1}-\frac 1{{\tt x}-r_2}\Bigr)\Bigr)^2\Bigr]\cr
&\kern 3.5cm ({\tt x}^2+{\tt x}-z=({\tt x}-r_1)({\tt x}-r_2),\ r_1,r_2\in\Bbb F_q,\ r_1\ne r_2)\cr
=\,&\text{ct}\Bigl[-\frac z{{\tt x}^{q+1}}\Bigl(\frac{2z}{r_1-r_2}\Bigl(\frac 1{{\tt x}-r_1}-\frac 1{{\tt x}-r_2}\Bigr)\cr
&+\frac {z^2}{(r_1-r_2)^2}\Bigl(\frac 1{({\tt x}-r_1)^2}+\frac 1{({\tt x}-r_2)^2}-\frac 2{({\tt x}-r_1)({\tt x}-r_2)}\Bigr)\Bigr)\Bigr]\cr
=\,&\text{ct}\Bigl[-\frac {z^2}{{\tt x}^{q+1}}\Bigl(\frac 2{r_1-r_2}\Bigl(\frac 1{{\tt x}-r_1}-\frac 1{{\tt x}-r_2}\Bigr)+\frac z{(r_1-r_2)^2}\Bigl(\frac 1{({\tt x}-r_1)^2}-\frac 1{({\tt x}-r_2)^2}\Bigr)\cr
&-\frac {2z}{(r_1-r_2)^2}\cdot \frac 1{r_1-r_2}\Bigl(\frac 1{{\tt x}-r_1}-\frac 1{{\tt x}-r_2}\Bigr)\Bigr)\Bigr]\cr
=\,&\text{ct}\Bigl[-\frac {z^2}{{\tt x}^{q+1}}\Bigl(\Bigl(\frac 2{r_1-r_2}-\frac {2z}{(r_1-r_2)^3}\Bigr)\Bigl(-\frac 1{r_1}\cdot\frac 1{1-\frac{\tt x}{r_1}}+\frac 1{r_2}\cdot\frac 1{1-\frac{\tt x}{r_2}}\Bigr)\cr
&+\frac z{(r_1-r_2)^2}\Bigl(\frac 1{r_1^2}\cdot\frac 1{(1-\frac{\tt x}{r_1})^2}+\frac 1{r_2^2}\cdot\frac 1{(1-\frac{\tt x}{r_2})^2}\Bigr)\Bigr)\Bigr]\cr
=\,&-z^2\biggl[\Bigl(\frac 2{r_1-r_2}-\frac {2z}{(r_1-r_2)^3}\Bigr)\Bigl(-\frac 1{r_1}\Bigl(\frac 1{r_1}\Bigr)^{q+1}+\frac 1{r_2}\Bigl(\frac 1{r_2}\Bigr)^{q+1}\Bigr)\cr
&+\frac z{(r_1-r_2)^2}\biggl(\frac 1{r_1^2}\binom{-2}{q+1}\Bigl(-\frac 1{r_1}\Bigr)^{q+1}+\frac 1{r_2^2}\binom{-2}{q+1}\Bigl(-\frac 1{r_2}\Bigr)^{q+1}\biggr)\biggr]\cr
=\,&-z^2\Bigl[\Bigl(\frac 2{r_1-r_2}-\frac {2z}{(r_1-r_2)^3}\Bigr)\Bigl(-\frac 1{r_1^3}+\frac 1{r_2^3}\Bigr)+\frac z{(r_1-r_2)^2}\Bigl(\frac 2{r_1^4}+\frac 2{r_2^4}\Bigr)\Bigr]\cr
&\kern 3.5cm \textstyle (\binom{-2}{q+1}\equiv 2\pmod p)\cr
=\,&-2z^2\Bigl[\Bigl(\frac 1{r_1-r_2}-\frac z{(r_1-r_2)^3}\Bigr)\frac{r_1^3-r_2^3}{(r_1r_2)^3}+\frac z{(r_1-r_2)^2}\cdot\frac{r_1^4+r_2^4}{(r_1r_2)^4}\Bigr]\cr
=\,&-2z^2\Bigl[\Bigl(1-\frac z{(r_1-r_2)^2}\Bigr)\frac{r_1^2+r_1r_2+r_2^2}{-z^3}+\frac 1{z^3}\cdot\frac{r_1^4+r_2^4}{(r_1-r_2)^2}\Bigr]\cr
=\,&-\frac 2z\Bigl[\Bigl(\frac z{(r_1-r_2)^2}-1\Bigr)(1+z)+\frac{(r_1^2-r_2^2)^2+2r_1^2r_2^2}{(r_1-r_2)^2}\Bigr]\cr
=\,&-\frac 2z\Bigl[\Bigl(\frac z{(r_1-r_2)^2}-1\Bigr)(1+z)+(r_1+r_2)^2+\frac{2z^2}{(r_1-r_2)^2}\Bigr]\cr
=\,&-\frac 2z\Bigl[\Bigl(\frac z{1+4z}-1\Bigr)(1+z)+1+\frac{2z^2}{1+4z}\Bigr]
=\frac{2z}{1+4z}.
\end{split}
\]
\end{proof}

\begin{proof}[Proof of Theorem~\ref{T1.1}] Recall again that in Theorem~\ref{T1.1}, $q$ is odd and $t\in\Bbb F_q^*$.

\medskip
($\Rightarrow$)
$1^\circ$ Let $\alpha=0$ and $\beta=q-1$ in \eqref{4.1}. We have
\[
\begin{split}
0\,&=\sum_{x\in\Bbb F_{q^2}}f(x)^{(q-1)q}\cr
&=-\sum_{\substack{j,k\ge 0\cr 1+j-2k=0,\,q+1}}\binom{q-1}j\binom jk(-1)^{k+j}t^{-j}\cr
&=-\sum_{\substack{j,k\ge 0\cr 0\le j\le q-1\cr 1+j-2k=0}}\binom jk(-1)^kt^{-j}\kern 2.8cm \textstyle(\binom{q-1}j\equiv(-1)^j\pmod p)\cr 
&=-\sum_{1\le k\le\frac q2}\binom{2k-1}k(-1)^kt^{1-2k}\cr
&=-\sum_{1\le k\le\frac q2}\binom{-k}kt^{1-2k}\cr
&=-t\biggl(\sum_{0\le k\le\frac q2}\binom{-k}k t^{-2k}-1\biggr).
\end{split}
\]
Thus
\begin{equation}\label{5.1}
\sum_{0\le k\le\frac q2}\binom{-k}k t^{-2k}=1.
\end{equation}
Put $z=t^{-2}$. Then by \eqref{5.1} and Lemma~\ref{L5.1}, ${\tt x}^2-{\tt x}-z$ has two distinct roots in $\Bbb F_q$, i.e., $1+4z$ is a square in $\Bbb F_q^*$.

\medskip

$2^\circ$ Let $\alpha=1$ and $\beta=q-2$ in \eqref{4.1}. We have
\[
\kern-4cm
\begin{split}
0
=\,&\sum_{x\in\Bbb F_{q^2}}f(x)^{1+(q-2)q}\cr
=\,&-\sum_{\substack{j,k\ge 0\cr 2+j-2k=0,\,q+1}}\binom{q-2}j\binom jk(-1)^{k+j}t^{-j}\cr
&-\sum_{\substack{j,k\ge 0\cr 3+j-2k=0,\,q+1}}\binom{q-2}j\binom {1+j}k(-1)^{k+j}t^{-1-j}\cr
=\,&-\sum_{1\le k\le \frac q2}\binom{q-2}{2k-2}\binom{2k-2}k(-1)^kt^{2-2k}\cr
&+\sum_{2\le k\le \frac {q+1}2}\binom{q-2}{2k-3}\binom{2k-2}k(-1)^kt^{2-2k}+1\cr
\end{split}
\]
\[
\begin{split}
=\,&-\sum_{1\le k\le \frac q2}(2k-1)\binom{-k+1}kt^{2-2k}-\sum_{2\le k\le \frac {q+1}2}(2k-2)\binom{-k+1}kt^{2-2k}+1\cr
&\kern 4.5cm \textstyle(\binom{q-2}{2k-2}\equiv 2k-1,\ \binom{q-2}{2k-3}\equiv-(2k-2)\pmod p)\cr
=\,&-\sum_{1\le k\le \frac {q+1}2}(4k-3)\binom{-k+1}kt^{2-2k}+1\cr
=\,&-\sum_{0\le k\le \frac {q-1}2}(4k+1)\binom{-k}{k+1}t^{-2k}+1\cr
=\,&-4\sum_{0\le k\le \frac {q-1}2}(k+1)\binom{-k}{k+1}z^k+3\sum_{0\le k\le \frac {q-1}2}(k+1)\binom{-k}{k+1}z^k+1\kern 0.6cm (z=t^{-2})\cr
=\,&-4\cdot \frac{2z}{1+4z}+3\cdot 1+1 \kern 1cm \text{(by Lemmas~\ref{L5.2}, \ref{L5.3})}\cr
=\,&4\cdot \frac{1+2z}{1+4z}.
\end{split}
\]
So we have $z=-\frac 12$, i.e., $t^2=-2$. By $1^\circ$, $1+4z=-1$ is a square in $\Bbb F_q^*$. Since both $-1$ and $-2$ are squares in $\Bbb F_q^*$, we have $q\equiv 1\pmod 8$.
\end{proof}


\section{The Polynomial $g_{n,q}$}

The polynomial $f=-{\tt x}+t{\tt x}^q+{\tt x}^{2q-1}$ considered in the present paper is closely related to another class of polynomials which we will discuss briefly in this section.

For each prime power $q$ and integer $n\ge 0$, the functional equation 
\[
\sum_{a\in\Bbb F_q}({\tt x}+a)^n=g_{n,q}({\tt x}^q-{\tt x})
\]
defines a polynomial $g_{n,q}\in\Bbb F_p[{\tt x}]$, where $p=\text{char}\,\Bbb F_q$. The polynomial $g_{n,q}$ was introduced in \cite{Hou11} as the $q$-ary version of the {\em reversed Dickson polynomial}. ($g_{n,2}$ is the reversed Dickson polynomial in characteristic $2$.) The rationale and incentive for studying the class $g_{n,q}$ are the fact that the class contains many interesting new PPs; see \cite{FHL,Hou12}. We are interested in triples of integers $(n,e;q)$ for which $g_{n,q}$ is a PP of $\Bbb F_{q^e}$, and we call such triples {\em desirable}. Among the known desirable triples $(n,e;q)$, $n$ frequently appears in the form $n=q^a-q^b-1$, $0<b<a<pe$.

Let us consider triples of the form 
\begin{equation}\label{6.1}
(q^a-q^b-1,e;q),\qquad e\ge 2,\ 0<b<a<pe.
\end{equation} 
(We assume $e\ge 2$ since all desirable triples with $e=1$ have been determined \cite[Corollary~ 2.2]{FHL}.) It is known that if $(a,b)=(2,1)$ and $\text{gcd}(q-2,q^e-1)=1$, or if $a\equiv b\equiv 0\pmod e$, then the triple \eqref{6.1} is desirable. It is also conjectured that the converse of the above statement is true for $e\ge 3$. When $e=2$, the situation becomes very interesting and complicated \cite[Section 5 and Table 1]{FHL}. It is known that for $q>2$ and $i>0$,
\[
g_{q^{2i}-q-1,q}({\tt x})\equiv (i-1){\tt x}^{q^2-q-1}-i{\tt x}^{q-2}\pmod{{\tt x}^{q^2}-{\tt x}}.
\]
Note that $g_{q^{2i}-q-1,q}({\tt x}^{q^2-q-1})\equiv(i-1){\tt x}-i{\tt x}^{2q-1}\pmod{{\tt x}^{q^2}-{\tt x}}$, where ${\tt x}^{q^2-q-1}$ is a PP of $\Bbb F_{q^2}$. So $g_{q^{2i}-q-1,q}$ is a PP of $\Bbb F_{q^2}$ if and only if $(i-1){\tt x}-i{\tt x}^{2q-1}$ is.
The attempt to determine the desirable triples of the form $(q^{2i}-q-1,2;q)$ has led to a more general result: In \cite{Hou}, all PPs of $\Bbb F_{q^2}$ of the form $t{\tt x}+{\tt x}^{2q-1}$, $t\in\Bbb F_q^*$, have been determined. For $q>2$ and $i>0$, we also have
\[
g_{q^{2i+1}-q-1,q}({\tt x})\equiv -{\tt x}^{q^2-2}+i{\tt x}^{q^2-q-1}-i{\tt x}^{q-2}\pmod{{\tt x}^{q^2}-{\tt x}}.
\]
Note that $g_{q^{2i+1}-q-1,q}({\tt x}^{q^2-q-1})\equiv i{\tt x}-{\tt x}^q-i{\tt x}^{2q-1}\pmod{{\tt x}^{q^2}-{\tt x}}$,
which is of course the prototype of the polynomial $f$ considered in the present paper. The following corollary immediately follows from Theorem~\ref{T1.1}. (Note that $i^2=-\frac 12$ has a solution in $\Bbb F_p$ if and only if $p\equiv 1$ or $3\pmod 8$.)

\begin{cor}\label{C6.1}
Let $q$ be odd, $i>0$, and $i\not\equiv 0\pmod p$. Then $(q^{2i+1}-q-1,2;q)$ is desirable if and only if $p\equiv 1$ or $3\pmod 8$, $q\equiv 1\pmod 8$, and $i^2=-\frac 12$.
\end{cor}

The corresponding corollary of Theorem~\ref{T1.2} is already known \cite[Theorem~5.9 (ii)]{FHL}.


\section{Some Congruence Questions}

Assume $q\equiv 1\pmod 8$ and $t\in\Bbb F_q^*$, $t^2=-2$. Then by Theorem~\ref{T1.1} and \eqref{4.1}, we have
\begin{equation}\label{7.1}
\sum_{\substack{i,j,k\ge 0\cr \alpha+1+i+j-2k=0,\,q+1}}\binom\alpha i\binom\beta j\binom{i+j}k(-1)^{k+j}t^{-(i+j)}=0
\end{equation}
for all integers $\alpha,\beta\ge 0$ with $\alpha+\beta=q-1$. In the left side of \eqref{6.1}, let 
\[
I_1=\sum_{\substack{i,j,k\ge 0\cr \alpha+1+i+j-2k=0}},\qquad I_2=\sum_{\substack{i,j,k\ge 0\cr \alpha+1+i+j-2k=q+1}}.
\]
We try to express $I_1$ in terms of $\beta$ only and $I_2$ in terms of $\alpha$ only. First, we have
\begin{equation}\label{7.2}
\begin{split}
I_1\,&=\sum_{\alpha+1\le k\le\frac{\alpha+q}2}\sum_{\substack{0\le i\le\alpha\cr 0\le j\le \beta\cr i+j=2k-\alpha-1}}\binom\alpha i\binom\beta j\binom{2k-\alpha-1}k(-1)^{k+j}t^{\alpha+1-2k}\cr
&=\sum_{\alpha+1\le k\le\frac{\alpha+q}2}\sum_{\substack{0\le i\le\alpha\cr 0\le j\le \beta\cr i+j=-2k+\alpha+q}}\binom\alpha i\binom\beta j\binom{2k-\alpha-1}k(-1)^{k+\beta+j}t^{\alpha+1-2k}\cr
&\kern 2.5cm (i\mapsto \alpha-i,\ j\mapsto \beta-j)\cr
&=\sum_{\alpha+1\le k\le\frac{\alpha+q}2}\sum_{\substack{i,j\ge 0\cr i+j=-2k+\alpha+q}}
\binom{\beta+i}\beta\binom\beta j\binom{2k-\alpha-1}{k-\alpha-1}(-1)^{k+1}t^{\alpha+1-2k}\cr
&\kern 2.5cm \textstyle (\binom\alpha i=(-1)^i\binom{-\alpha-1+i}i\equiv(-1)^i\binom{\beta+i}i=(-1)^i\binom{\beta+i}\beta\pmod p)\cr
&=\sum_{0\le k\le\frac{\beta-1}2}\sum_{\substack{i,j\ge 0\cr i+j=\beta-2k-1}}
\binom{\beta+i}\beta\binom\beta j\binom{2k-\beta}k(-1)^{k+\alpha}t^{-\alpha-1-2k}\cr
&\kern 2.5cm (k\mapsto k+\alpha+1)\cr
&=\sum_{0\le k\le\frac{\beta-1}2}\,\sum_{0\le j\le \beta-2k-1}\binom{2\beta-2k-j-1}\beta\binom\beta j\binom{2k-\beta}k(-1)^{k+\beta}t^{\beta-1-2k}.
\end{split}
\end{equation}
We also have
\begin{equation}\label{7.3}
\begin{split}
I_2\,&=\sum_{0\le k\le\frac{\alpha-1}2}\sum_{\substack{0\le i\le\alpha\cr 0\le j\le \beta\cr i+j=2k-\alpha+q}}\binom\alpha i\binom\beta j\binom{2k-\alpha}k(-1)^{k+j}t^{\alpha-1-2k}\cr
&=\sum_{0\le k\le\frac{\alpha-1}2}\sum_{\substack{0\le i\le\alpha\cr 0\le j\le \beta\cr i+j=-2k+\alpha-1}}\binom\alpha i\binom\beta j\binom{2k-\alpha}k(-1)^{k+\beta+j}t^{\alpha-1-2k}\cr
&\kern6.8cm (i\mapsto \alpha-i,\ j\mapsto\beta-j)\cr
&=\sum_{0\le k\le\frac{\alpha-1}2}\sum_{\substack{i,j\ge 0\cr i+j=-2k+\alpha-1}}\binom\alpha i\binom{\alpha+j}\alpha\binom{2k-\alpha}k(-1)^{k+\beta}t^{\alpha-1-2k}\cr
&\kern 6.8cm \textstyle (\binom\beta j\equiv(-1)^j\binom{\alpha+j}\alpha\pmod p)\cr
&=\sum_{0\le k\le\frac{\alpha-1}2}\,\sum_{0\le i\le\alpha-2k-1}\binom\alpha i\binom{2\alpha-2k-i-1}\alpha\binom{2k-\alpha}k(-1)^{k+\alpha}t^{\alpha-1-2k}.
\end{split}
\end{equation}
For integer $a\ge 0$, define
\begin{equation}\label{7.4}
H(a)=(-2)^{\lfloor\frac{a-1}2\rfloor}\sum_{0\le k\le\frac{a-1}2}\,\sum_{0\le i\le a-2k-1}\binom ai\binom{2k-a}k\binom{2a-2k-i-1}a2^{-k}.
\end{equation}
Then by \eqref{7.2},
\begin{equation}\label{7.5}
I_1=(-1)^\beta t^{\beta-1}(-2)^{-\lfloor\frac{\beta-1}2\rfloor}H(\beta)=(-1)^\beta t^{2\{\frac{\beta-1}2\}}H(\beta),
\end{equation}
where $\{x\}=x-\lfloor x\rfloor$. By \eqref{7.3}, in the same way,
\begin{equation}\label{7.6}
I_2=(-1)^\alpha t^{2\{\frac{\alpha-1}2\}}H(\alpha).
\end{equation}
Therefore we may restate \eqref{7.1} as the following corollary.

\begin{cor}\label{C7.1}
Assume that $q$ is a power of a prime $p$ with $q\equiv 1\pmod 8$. Then for integers $\alpha,\beta\ge 0$ with $\alpha+\beta=q-1$, we have
\begin{equation}\label{7.7}
H(\alpha)+H(\beta)\equiv 0\pmod p.
\end{equation}
\end{cor}

\noindent{\bf Question.} Is there a more direct number theoretic proof for \eqref{7.7}?

\medskip

It is possible to write the function $H(a)$ as an unrestricted double sum. We have
\begin{equation}\label{7.8}
\begin{split}
H(a)\,&=(-2)^{\lfloor\frac{a-1}2\rfloor}\sum_{0\le k\le\frac{a-1}2}\sum_{i\ge 0}\binom ai\binom{2k-a}k\binom{2a-2k-i-1}a2^{-k}\cr
&\kern 1.2cm \textstyle (\binom{2a-2k-i-1}a=0\ \text{when $0\le k\le\frac{a-1}2$ and $a-2k-1<i\le a$})\cr
&=(-2)^{\lfloor\frac{a-1}2\rfloor}\sum_{0\le k\le\frac{a-1}2}\sum_{i\ge 0}\binom ai\binom{a-1-k}k\binom{2a-2k-i-1}a(-2)^{-k}\cr
&=(-2)^{\lfloor\frac{a-1}2\rfloor}\sum_{0\le k\le\frac{a-1}2}\sum_{i\ge 0}\binom ai\binom{a-1-k}{a-1-2k}\binom{a-1-2k+i}a(-2)^{-k}\cr
&\kern 1.2cm (i\mapsto a-i)\cr
&=(-2)^{\lfloor\frac{a-1}2\rfloor}\sum_{\substack{0\le l\le a-1\cr l\equiv a-1\,\text{(mod 2)}}}\sum_{i\ge 0}\binom ai\binom{\frac{a-1+l}2}l\binom{l+i}a(-2)^{-\frac 12(a-1-l)}\cr
&\kern 1.2cm (l=a-1-2k)\cr
&=(-2)^{\lfloor\frac{a-1}2\rfloor}\sum_{\substack{l\ge 0\cr l\equiv a-1\,\text{(mod 2)}}}\sum_{i\ge 0}\binom ai\binom{\frac{a-1+l}2}l\binom{l+i}a(-2)^{-\frac 12(a-1-l)}\cr
&\kern 1.2cm \textstyle (\binom{\frac{a-1+l}2}l=0\ \text{when $l>a-1$ and $l\equiv a-1\pmod 2$}) \cr
&=\sum_{j,i\ge 0}\binom ai \binom{\lfloor\frac a2\rfloor+j}{2j+2\{\frac{a-1}2\}} \binom{2j+2\{\frac{a-1}2\}+i}a (-2)^j\cr
&\kern 1.2cm \textstyle (l=2j+2\{\frac{a-1}2\}).
\end{split}
\end{equation}

\medskip

\noindent{\bf Remark.} Each of the sequences $H(2n)$ and $H(2n+1)$ satisfies a recurrence relation of order $4$. See the appendix for the details.

\medskip

We now discuss another strange phenomenon associated with $H(a)$. In the final expression of \eqref{7.8}, make a change of variable $i\mapsto a-i$. Then we have
\begin{equation}\label{7.9}
\begin{split}
H(a)\,&=\sum_{j,i\ge 0}\binom ai \binom{\lfloor\frac a2\rfloor+j}{2j+2\{\frac{a-1}2\}} \binom{a+2j+2\{\frac{a-1}2\}-i}a (-2)^j\cr
&=\sum_{0\le j\le\frac{a-1}2}\;\sum_{0\le i\le 2j+2\{\frac{a-1}2\}}\binom ai \binom{\lfloor\frac a2\rfloor+j}{2j+2\{\frac{a-1}2\}} \binom{a+2j+2\{\frac{a-1}2\}-i}{2j+2\{\frac{a-1}2\}-i} (-2)^j.
\end{split}
\end{equation}
Assume $q$ is odd. For $0\le a\le q-1$, we obtain below an expression for $H(q-1-a)$ in $\Bbb Z_p/p\Bbb Z_p$ ($=\Bbb F_p$) that is almost identical to \eqref{7.9}. We have 
\eject
\ \vskip-5mm
\begin{equation}\label{7.10}
\begin{split}
&H(q-1-a)\cr
=\,&\sum_{0\le j\le\frac{q-a-2}2}\;\sum_{0\le i\le 2j+2\{\frac{a-1}2\}}\binom{q-1-a}i\binom{\lfloor\frac{q-1-a}2\rfloor+j}{2j+2\{\frac{a-1}2\}}\binom{q-1-a+2j+2\{\frac{a-1}2\}-i}{2j+2\{\frac{a-1}2\}-i}(-2)^j\cr
&\kern 8cm \text{(by \eqref{7.9})}\cr
=\,&\sum_{0\le j\le\frac{q-a-2}2}\;\sum_{0\le i\le 2j+2\{\frac{a-1}2\}}\binom{-1-a}i\binom{-\lceil\frac a2\rceil-\frac 12+j}{2j+2\{\frac{a-1}2\}}\binom{-1-a+2j+2\{\frac{a-1}2\}-i}{2j+2\{\frac{a-1}2\}-i}(-2)^j\cr
=\,&\sum_{0\le j\le\frac{q-a-2}2}\;\sum_{0\le i\le 2j+2\{\frac{a-1}2\}}(-1)^i\binom{a+i}i(-1)^{a-1}\binom{\lceil\frac a2\rceil +2\{\frac{a-1}2\}+j-\frac 12}{2j+2\{\frac{a-1}2\}}\cr
&\cdot (-1)^{a-1-i}\binom a{2j+2\{\frac{a-1}2\}-i}(-2)^j\cr
=\,&\sum_{0\le j\le\frac{q-a-2}2}\;\sum_{0\le i'\le 2j+2\{\frac{a-1}2\}}\binom a{i'}\binom{\lfloor\frac a2\rfloor +\frac 12+j}{2j+2\{\frac{a-1}2\}}\binom{a+2j+2\{\frac{a-1}2\}-i'}{2j+2\{\frac{a-1}2\}-i'}(-2)^j\cr
&\kern 8cm \textstyle (i'=2j+2\{\frac{a-1}2\}-i).
\end{split}
\end{equation}

We claim that when $\frac{q-1}2\le a\le q-1$, in the last expression of \eqref{7.10}, the sum $\sum_{0\le j\le\frac{q-a-2}2}$ may be replaced with $\sum_{0\le j\le\frac{a-1}2}$. In fact, we have $\frac{a-1}2\ge\frac{q-a-2}2$. For $\frac{q-a-2}2<j\le \frac{a-1}2$, we have 
\[
\begin{split}
\binom{\lfloor\frac a2\rfloor +\frac 12+j}{2j+2\{\frac{a-1}2\}}\,&=(-1)^{a-1} \binom{-\lfloor\frac a2\rfloor +2\{\frac{a-1}2\}-\frac 32+j}{2j+2\{\frac{a-1}2\}}\cr
&=(-1)^{a-1} \binom{-\frac 12-\lceil\frac a2\rceil+j}{2j+2\{\frac{a-1}2\}}\cr
&\equiv (-1)^{a-1} \binom{\frac {q-1}2-\lceil\frac a2\rceil+j}{2j+2\{\frac{a-1}2\}}\pmod p.
\end{split}
\]
Since $\frac{q-1}2-\lceil\frac a2\rceil+j\in\Bbb Z$ and $0\le\frac{q-1}2-\lceil\frac a2\rceil+j<2j+2\{\frac{a-1}2\}$, the above binomial coefficient is $0$. Hence the claim is proved.

By \eqref{7.9}, \eqref{7.10}, and the above claim, for $\frac{q-1}2\le a\le q-1$, we have
\begin{equation}\label{7.11}
\begin{split}
H(a)+H(q-1-a)
=\,&\sum_{0\le j\le\frac{a-1}2}\;\sum_{0\le i\le 2j+2\{\frac{a-1}2\}}(-2)^j\binom ai\binom{a+2j+2\{\frac{a-1}2\}-i}{2j+2\{\frac{a-1}2\}-i}\cr
&\cdot\biggl[\binom{\lfloor\frac a2\rfloor+j}{2j+2\{\frac{a-1}2\}}+\binom{\lfloor\frac a2\rfloor+\frac 12+j}{2j+2\{\frac{a-1}2\}}\biggr]
\end{split}
\end{equation}
in $\Bbb Z_p/p\Bbb Z_p$.

Let $S(a)$ denote the right side of \eqref{7.11} for each integer $a\ge 0$. Hence $S$ is a function from the set of nonnegative integers to $\Bbb Q$. It is interesting to observe that $S(a)$ is independent of $q$ and that Corollary~\ref{C7.1} has the following equivalent form.

\begin{cor}\label{C7.2}
Assume $q\equiv 1\pmod 8$. Then
\begin{equation}\label{7.12}
S(a)\equiv 0\pmod p\qquad \text{for}\ \frac {q-1}2\le a\le q-1.
\end{equation}
\end{cor}

Equation~\eqref{7.12} is quite curious. In fact, 
computer experiments suggest a congruence pattern for the function $S$ that is even stronger than \eqref{7.12}. We conclude this section with the following conjecture.

\begin{conj}\label{CJ7.3} We have
\[
S(a)\equiv 
\begin{cases}
0\pmod 3&\text{if $a=0$ or $3^{2k+1}\le a\le 3^{2k+2}-1$ for some integer $k\ge 0$},\cr
-1\pmod 3&\text{otherwise}.
\end{cases}
\]
\end{conj}

\bigskip

\appendix

\section{Recurrence Relations for $H(2n)$ and $H(2n+1)$}

We write \eqref{7.8} as
\[
\begin{split}
H(2n)\,&=\sum_{j,i\ge 0}\binom{2n}i\binom{n+j}{2j+1}\binom{2j+1+i}{2n}(-2)^j,\cr
H(2n+1)\,&=\sum_{j,i\ge 0}\binom{2n+1}i\binom{n+j}{2j}\binom{2j+i}{2n+1}(-2)^j,
\end{split}
\]
where $n\ge 0$. Using the Zeilberger algorithm for hypergeometric multi sums \cite{Apa-Zei,WZ92a}, we find that each of the sequences $H(2n)$ and $H(2n+1)$ satisfies a recurrence relation of order 4. For any sequence $F(n)$, define $NF(n)=F(n+1)$. Then $H(2n)$ and $H(2n+1)$ are annihilated by the operators in (A1) and (A2), respectively. 

\newpage
{\small
\begin{sideways}
\hbox to \vsize{(A.1)\hfill}
\end{sideways}
\hspace{2mm}
\begin{sideways}
\hbox to \vsize{\kern0mm
$(3+n) (4+n) (5+2 n) (7+2 n) (6123+15568 n+14248 n^2+5600 n^3+800 n^4) N^4$\hfill
}
\end{sideways}
\begin{sideways}
\hbox to \vsize{\kern0mm
$+ 8 (3+n) (5+2 n) (1662798+5262881 n+6660272 n^2+4334056 n^3+1534248 n^4+280800 n^5+20800 n^6) N^3$\hfill
}
\end{sideways}
\begin{sideways}
\hbox to \vsize{\kern0mm
$+ 2 (1777193487+7186386537 n+12405568049 n^2+11978835516 n^3+7092508484 n^4+2641104576 n^5+604784576 n^6+77932800 n^7+4329600 n^8) N^2$\hfill
}
\end{sideways}
\begin{sideways}
\hbox to \vsize{\kern0mm
$+ 8 (2+n) (3+2 n) (1610409+5182411 n+6622148 n^2+4328408 n^3+1534248 n^4+280800 n^5+20800 n^6) N$\hfill
}
\end{sideways}
\begin{sideways}
\hbox to \vsize{\kern0mm
$+ (1+n) (2+n) (1+2 n) (3+2 n) (42339+64064 n+35848 n^2+8800 n^3+800 n^4)$\hfill
}
\end{sideways}
\hspace{6mm}
\begin{sideways}
\hbox to \vsize{(A.2)\hfill}
\end{sideways}
\hspace{2mm}
\begin{sideways}
\hbox to \vsize{\kern0mm
$(3+n) (4+n) (7+2 n) (9+2 n) (18219+34416 n+23848 n^2+7200 n^3+800 n^4) N^4$\hfill
}
\end{sideways}
\begin{sideways}
\hbox to \vsize{\kern0mm
$+ 8 (3+n) (7+2 n) (6606054+16032469 n+15833228 n^2+8156552 n^3+2314248 n^4+343200 n^5+20800 n^6) N^3$\hfill
}
\end{sideways}
\begin{sideways}
\hbox to \vsize{\kern0mm
$+ 2 (10505025027+33069872253 n+44933987909 n^2+34456630164 n^3+16323110084 n^4+4894912704 n^5+907856576 n^6+95251200 n^7+4329600 n^8) N^2$\hfill
}
\end{sideways}
\begin{sideways}
\hbox to \vsize{\kern0mm
$+ 8 (2+n) (5+2 n) (6503193+15909639 n+15786632 n^2+8150904 n^3+2314248 n^4+343200 n^5+20800 n^6) N$\hfill
}
\end{sideways}
\begin{sideways}
\hbox to \vsize{\kern0mm
$+ (1+n) (2+n) (3+2 n) (5+2 n) (84483+106912 n+50248 n^2+10400 n^3+800 n^4)$\hfill
}
\end{sideways}
}

\newpage

\end{document}